\documentclass[oneside,leqno,11pt]{amsart}
\pdfoutput=1
\usepackage{cleveref}
\usepackage{amsmath, amsthm, amssymb, amsfonts}
\usepackage[textwidth=16cm,textheight=20cm]{geometry}
\usepackage{bbm}
\usepackage{enumitem}
\usepackage{color}

\newtheorem{thm}[equation]{Theorem}

\newtheorem{lem}[equation]{Lemma}

\newtheorem{prop}[equation]{Proposition}
\theoremstyle{definition}

 \numberwithin{equation}{section}

\def\P{{\mathbb{P}}}
\def\N{{\mathbb{N}}}
\def\R{{\mathbb{R}}}
\def\E{{\mathbb{E}}}
\def\8{\infty}
\renewcommand{\a}{\alpha}

\begin{document}

\title{Tail asymptotics  of maximums on trees in the critical case}
\author[M. Ma\'slanka]{Mariusz Ma\'slanka}
\address{M. Ma\'slanka\\ Instytut Matematyczny\\ Uniwersytet Wroclawski\\ 50-384 Wroclaw\\
pl. Grunwaldzki 2/4\\ Poland}
\email{maslanka@math.uni.wroc.pl}



\thanks{The research was partially supported by the National Science Centre, Poland (grant number UMO-2014/15/B/ST1/00060).}

\begin{abstract}
We consider solutions to the maximum recursion on weighted branching trees given by$$X\,{\buildrel d\over=}\,\bigvee_{i=1}^{N}{A_iX_i}\vee B,$$where $N$ is a random natural number, $B$ and $\{A_i\}_{i\in\mathbb{N}}$ are random positive numbers and $X_i$ are independent copies of $X$, also independent of $N$, $B$, $\{A_i\}_{i\in\mathbb{N}}$. Properties of solutions to this equation are governed mainly by the function $m(s)=\mathbb{E}\big[\sum_{i=1}^NA_i^s\big]$. Recently, Jelenkovi\'c and Olvera-Cravioto proved, assuming e.g. $m(s)<1$ for some $s$, that the asymptotic behavior of the endogenous solution $R$ to the above equation is power-law, i.e.$$\mathbb{P}[R>t]\sim Ct^{-\alpha}$$for some $\alpha>0$ and $C>0$. In this paper we assume $m(s)\ge 1$ for all $s$ and prove analogous results.
\end{abstract}

\maketitle
\section{Introduction}
In this paper we study the maximum recursion on trees
\begin{equation}\label{eq:1}
X \,{\buildrel d \over =}\, \bigvee_{i=1}^{N}{A_i X_i} \vee B,
\end{equation}
where $N$ is a random natural number, $B$ and $\{A_i\}_{i\in\N}$ are random positive numbers and $X_i$ are independent copies of $X$, which are independent also of $N$, $B$, $\{A_i\}_{i\in \N}$. Our main goal is to describe asymptotic properties of the endogenous solution to \eqref{eq:1} (in the sense of \cite{AB}).

Observe that for $N=1$ a.s. equation \eqref{eq:1} is just the random extremal equation considered by Goldie \cite{G}. Moreover, in this case, taking logarithm of both sides of the equation, we obtain the classical Lindley's equation related to the reflected random walk. In general, equation \eqref{eq:1} is called high-order Lindley equation and is a useful  tool in studying branching random walks. We refer to  Aldous, Bandyopadhyay \cite{AB} and Jelenkovi\'c, Olvera-Cravioto \cite{JOC}
for a more complete bibliography on the subject and description of a class of other related stochastic equations.

We begin with explaining how to construct an endogenous solution to equation \eqref{eq:1}. Let $\mathcal{T} =\bigcup_{k \geq 0} \mathbb{N}^k$ be an infinite Ulam-Harris tree, where $\mathbb{N}^0 = \{ \varnothing \}$. For $v = (i_1, . . . , i_n)$ we define the length $|v| = n$ and by $vi$ we denote the vertex $(i_1, i_2, . . . , i_n, i)$. We write $u < v$ if $u$ is a proper prefix of $v$, i.e. $u = (i_1, .., i_k)$ for some $k < n$. Moreover we write $u \leq v$ if $u < v$ or $u = v$. Now we take $\{(N(v),B(v), A_1(v), A_2(v), . . .)\}_{v \in \mathcal{T}}$ a family of i.i.d. copies of $(N,B, A_1, A_2, . . .)$ indexed by the vertices of $\mathcal{T}$. Since equation \eqref{eq:1} depends only on $N$ first values of $A_i$'s, we can assume that $A_i(v)=0$ for every $v\in\mathcal{T}$ and $i> N(v)$.
For $v \in \mathcal{T}$ we also define a random variable $L(\varnothing) = 1$ and $L(vi) = L(v)A_i(v)$. We define
\begin{equation}\label{eq:2}
R = \bigvee_{v \in \mathcal{T}}{L(v)B(v)}.
\end{equation}
One can easily deduce that if
 the maximum above is finite almost surely then the random variable $R$ satisfies \eqref{eq:1}.

The properties of $R$ are governed by the function
$$
m(s) = \E \left[ \sum_{i=1}^N A_i^s \right].
$$ Jelenkovi\'c and Olvera-Cravioto \cite{JOC} recently studied existence and asymptotic properties of $R$ in the case when the equation $m(s)=1$ has two solutions $\a<\beta$ and $\a<1$. They proved, under a number of further assumptions, that $R$ has a power-law distribution of order $\beta$, i.e.
$$
\P[R>t] \sim C t^{-\beta}, \qquad t\to\infty
$$ for some $C>0$. In this paper we consider the critical case, when the equation $m(s)=1$ has exactly one solution $\a$ and then $m'(\a)=0$. Our main result is the following.
\begin{thm}\label{th1}
Suppose that
\begin{enumerate}[label=(A{\arabic*})]
\item \label{st:0} $\P\left[B > 0 \right] > 0$ ,
\item \label{st:1} There exists $\a$ such that $m(\a) = \E\left[ \sum_{i=1}^N{A_i^{\a}}\right] = 1 $,
\item \label{st:2} $m'(\a) = \E\left[ \sum_{i=1}^N{A_i^{\a} \log{A_i}} \right] = 0$,
\item \label{st:3}$\E[N] > 1$,
\item \label{st:4}For some $j$ the measure $\P[\log{A_j} \in du, A_j >0, N\geq j]$ is non-arithmetic,
\item \label{st:5}$\E\left[B^{\a +\delta} + N^{1+\delta} + \sum_{i=1}^N \left(A_i^{-\delta } + A_i^{\a + \delta} \right) \right] < \infty $, for some $\delta > 0$.
\end{enumerate}

Then the solution $R$ of \eqref{eq:1} given by \eqref{eq:2} is well defined and

$$
\lim_{t \to \infty} t^{\a} \P[R > t] = C,
$$
for some constant  $ C >0$.
\end{thm}

Equation \eqref{eq:1} is similar to the linear stochastic equation (called also the smoothing transform)
\begin{equation}\label{eq:smooth}
X \,{\buildrel d \over =}\, \sum_{i=1}^{N}{A_i X_i} + B,
\end{equation}
  where $X_i$ are independent copies of $X$, which are also independent of a given sequence of non-negative random variables $(N,B,A_1,A_2,\ldots)$. This equation was considered in a number of papers, see e.g. \cite{bur-spa,BK,DL,JOC0,L, L2}. In these  papers existence and some further properties, including asymptotic behavior, of solutions to \eqref{eq:smooth} were considered.
 In particular, the techniques described there can be applied in our settings to study equation \eqref{eq:1}. The details will be given in next sections. 

\medskip

Our model is closely related to the branching random walks, which can be defined as follows. 
An initial ancestor is located at the origin. 
Its $N$ children, the first
generation,  are placed in $\R$ according to the distribution of the point process $\Theta =\{-\log A_i\}_{i=1}^N$, where $N$ and $\{A_i\}$ are as in \eqref{eq:1}. Each of the particles  produces
its own children which  are positioned (with
respect to their parent) according to the same distribution of $\Theta$ and they form the second generation. And so on. The resulting system
is called a branching random walk.

Notice that if $B =1$, then $M = -\log R$  describes the global minimum of  the branching random walk, that is the leftmost position of all the particles in the system. Thus our Theorem \ref{th1} implies that 
$$
\lim_{t \to \infty} e^{\a t} \P[M < -t] = C
$$ and $C>0$. The same result, however under much weaker hypotheses and using different techniques based on the spinal decomposition, was recently proved by Madaule \cite{M}.

\section{Upper and lower estimates of $R$}
The aim of this section is to provide upper and lower estimates for $R$ defined in \eqref{eq:2}, that is to prove under assumptions given in Theorem        \ref{th1}:
$$
 \frac{1}{C}  t^{-\a} \le \P[R>t] \le C t^{-\a}
$$ for some constant $C >0$ and large $t$. Notice that the upper estimate implies in particular that $R$ given by \eqref{eq:2} is finite a.s.

Before giving proofs we recall a useful tool, called the many-to-one formula.  First, let us introduce a random variable $Y$ with distribution given by
\begin{equation}\label{eq:3}
\E[f(Y)] = \E\left[ \sum_{i=1}^N f(-\log{A_i}) A_i^{\a} \right],
\end{equation}
for any positive Borel function $f$. By \ref{st:1} the right hand side of the above defines a probability measure. Moreover \ref{st:2}, \ref{st:4} and \ref{st:5} imply that
the random variable $Y$ is centered, non-arithmetic  and has finite exponential moments, i.e.
\[
\E\left[ e^{\pm \delta Y} \right] < \infty,
\]
for some $\delta > 0$.

Now, let {$\{Y_i\}$} be a sequence of independent copies of $Y$ defined by \eqref{eq:3} and let $S_n$ be the sequence of their partial sums, $S_n = \sum_{k=1}^nY_k$. For a fixed $n$ and any test function $f: \mathbb{R}^n \to \mathbb{R}$, the following  many-to-one formula holds
\begin{equation}\label{eq:5}
\E\left[e^{\a S_n} f(S_1,...,S_n) \right] = \E\bigg[ \sum_{|v| = n} f(-\log{L(v_1)},...,-\log{L(v_n)}) \bigg],
\end{equation}
see e.g. Theorem 1.1 in Shi \cite{Shi}.

\begin{lem}\label{prop1}
Under the assumptions of Theorem \ref{th1} there is a strictly positive constant $C$ such that for large $t$
$$
\frac{1}{C}  \leq t^{\a} \P[R > t] \leq C.
$$
\end{lem}
\begin{proof}
We apply here estimates proved by Buraczewski and Kolesko \cite{BK} in a slightly different settings. The authors considered tails of fixed points of the inhomogeneous smoothing transform, that is solutions
to the stochastic equation \eqref{eq:smooth}. 'Inhomogeneous' means that the $B$-term do not reduce to 0.
 All solutions to this equation  were described by Alsmeyer and Meiners \cite{AM}. In particular the endogenous  solution is given by
 $$\widetilde{R} = \sum_{v \in \mathcal{T}} L(v)B(v),$$
  assuming that the above series is finite a.s. It is known that if $\a<1$, then under hypotheses of Theorem \ref{th1}, the random variable $\widetilde R$ is finite a.s.  (see \cite{BK}, Proposition 2.1) and moreover
\begin{equation}\label{eq:4}
\P[\widetilde{R}>t] \sim \widetilde{C} t^{-\a},
\end{equation}
for some $\widetilde{C} > 0$, see \cite{BK}, Theorem 1.1.

We will also need that for any strictly positive constant $\delta$, the function
\begin{equation}
\label{eq:W}
W(x) = \E\left[ \sum_{i=0}^{\infty} e^{-\delta(x+S_i)} \mathbbm{1}(S_j + x \geq 0 \enskip \text{for} \enskip j \leq i)  \right]
\end{equation}
is bounded, see \cite{BK}, Lemma 2.2.

Since the results in \cite{BK} are proved only for $\a<1$ we split the proof into two cases. First we assume that $\a<1$ and we apply directly the results stated above. Next we reduce the general situation to this case.
\medskip

{\sc Case 1.} Assume that  $\a<1$ and choose $\delta<1-\a$.

For the upper bound simply note that 
$$
\P\left[R > t \right] \leq \P\left[\widetilde{R} > t \right],
$$
and the desired estimates on right hands side come from \eqref{eq:4}. 

Lower estimates are  more difficult to prove and usually require some tricky arguments. Below we present a proof based on the result by Buraczewski and Kolesko \cite{BK}, who studied the linear stochastic equation  \eqref{eq:smooth}. However, for reader's convenience in Appendix \ref{appendix} we present a complete proof, borrowed from A\"{i}d\'{e}kon \cite{A}, based on the second moment method. To simplify the arguments we write it for a very particular case when $N$ is constant and $B=1$ a.s. (for a proof when $N$ is random see Madaule \cite{M}).\\\\
Here, we proceed in two steps.\\
{\sc Step (i).} First, let us assume that $B =1$ a.s.

For large $M > 0$, whose precise value will be specified below, we write

$$
\P\left[\widetilde{R} > M t \right] \leq \P\left[R > t \right] + \P\left[ \{R \leq t \} \cap \{\widetilde{R} > M t \} \right].
$$
Taking $\gamma = \a + \delta $, we have

\begin{equation*}
\begin{split}
\P\left[ \{R \leq t \} \cap \{\widetilde{R} > M t \} \right]& \leq \P\left[ \sum_{v \in \mathcal{T}} L(v) \mathbbm{1}(L(v') \leq t \text{ for } v' \leq v) > M t \right]\\
&\leq  \P\left[ \sum_{v \in \mathcal{T}} L^{\gamma}(v) \mathbbm{1}(L(v') \leq t \text{ for } v' \leq v) > M^{\gamma}  t^{\gamma} \right] \\
& \leq M^{-\gamma} t^{-\gamma} \E\left[ \sum_{v \in \mathcal{T}} L^{\gamma}(v) \mathbbm{1}(L(v') \leq t \text{ for } v' \leq v) \right].
\end{split}
\end{equation*}
Using the many-to-one formula \eqref{eq:5} we obtain
\begin{equation*}
\begin{split}
\E\left[ \sum_{v \in \mathcal{T}} L^{\gamma}(v) \mathbbm{1}(L(v') \leq t \text{ for } v' \leq v) \right] & = \sum_n \E\left[ \sum_{|v|=n} L^{\gamma}(v) \mathbbm{1}(L(v') \leq t \text{ for } v' \leq v) \right]\\
& = \sum_n \E\left[e^{\a S_n} e^{-\gamma S_n} \mathbbm{1}(S_k + \log{t} \geq 0 \text{ for } k \leq n) \right]\\
& = \sum_n \E\left[ e^{-\delta (S_n + \log{t})} t^{\delta} \mathbbm{1}(S_k + \log{t} \geq 0 \text{ for } k \leq n) \right]\\
& = t^{\delta} W(\log{t}),
\end{split}
\end{equation*}
for the bounded function W defined in \eqref{eq:W}. The above implies
$$
\P\left[\widetilde{R} > M t \right] \leq \P\left[R > t \right] + C_1 M^{-\gamma} t^{-\a}.
$$
On the other hand, by \eqref{eq:4}, we have the lower estimate
$$
\P\left[\widetilde{R} > M t \right] > C_2 M ^{-\a} t^{-\a},
$$
for some $C_2 > 0$ and sufficiently large $t$. Therefore, taking $M$ big enough,  we can find $C > 0$ such that
\begin{equation}\label{eq:7}
\P\left[R > t\right] > C  t^{-\a}.
\end{equation}\\
{\sc Step (ii).} We now consider general $B$. For this purpose we define $R' = \bigvee_{v \in \mathcal{T}}{{L}}(v)$. For any $M >0$ write
$$
\P \left[R' > t M\right]  \leq  \P \left[R > t \right] + \P \left[R' > t M, R \leq t \right].
$$
We apply here similar arguments as in \cite{BK} (Proposition 2.1).  Putting $f(x_1,...,x_n) = \mathbbm{1}(x_1 \geq -\log{t},...,x_{n-1} \geq -\log{t}, x_n < -\log{t})$ in the many-to-one formula \eqref{eq:5} we obtain
\begin{equation*}
\begin{split}
\P\! [R' > t M,& R \leq t ] \\
&\leq \P\!\left[L(v) > t M, L(v)B(v) \leq t \text{ for some } v \text{ and } L(u) \leq t M \text{ for } u < v \right]\\
& \leq \P\!\!\left[\sum_{v \in \mathcal{T}}\!\mathbbm{1}(L(v) > t M \text{ and } L(u) \leq t M \text{ for } u < v) \mathbbm{1}(B(v) \leq M^{-1})\! \geq 1  \right]\\
& \leq \P[B \leq M^{-1}] \sum_{n} \E \left[\sum_{|v| = n}\mathbbm{1}(L(v) > t M  \text{ and } L(u) \leq t M \text{ for } u < v) \right]\\
& = \P[B \leq M^{-1}] \sum_{n} \E \left[e^{\a S_n} \mathbbm{1}(S_n < -\log (t M) \text{ and } S_k \geq -\log (t M) \text{ for } k < n) \right]\\
& \leq \P[B \leq M^{-1}] t^{-\a} M^{-\a}.
\end{split}
\end{equation*}

From the discussion in the first step there is $C>0$ such that
$$
Ct^{-a}M^{-a} \leq \P \left[R' > t M\right]  \leq  \P \left[R > t \right] + \P[B \leq M^{-1}] t^{-\a} M^{-\a},
$$
hence
$$
 \P \left[R > t \right] \geq t^{-a}M^{-a} (C - \P[B \leq M^{-1}]),
$$
and by \ref{st:0} we can take large $M$ to ensure $\P[B \leq M^{-1}] < C$.

{\sc Case 2.} We now consider $\a\ge 1$. Take any $\a_0, \delta_0$ such that $0 < \a_0 + \delta_0 < 1$ and $\frac{\a \delta_0}{\a_0} < \delta$. Define $\left(\overline{B},\overline{A}_1, \overline{A}_2,...  \right) = \left(B^{\a / \a_0},A_1^{\a / \a_0}, A_2^{\a / \a_0},...  \right)$ and $\overline{R} = \bigvee_{v \in \mathcal{T}}{\overline{L}(v) \overline{B}}(v)$, where $\overline{L}(v)$ is defined analogously to $L(v)$ but using new weights $\overline{A}_i$. 
We write
$$
\P\left[R > t \right] = \P\left[ \bigvee_{v \in \mathcal{T}}L(v)B(v) > t \right] = \P\left[ \bigvee_{v \in \mathcal{T}}\left( \overline{L}(v)\overline{B}(v)\right)^{\a_0 / \a} > t \right] = \P\left[ \bigvee_{v \in \mathcal{T}}{\overline{L}}(v){\overline{B}}(v)  > t^{\a / \a_0} \right],
$$
and the right hand side of the above is properly bounded by arguments given in the first case to the random variable $\overline{R}$.
\end{proof}

\section{Asymptotics of $R$}
To prove the precise asymptotic of $R$ we adopt to our settings the arguments presented by Durrett and Liggett \cite{DL} (see also \cite{BB,BK}), where the problem was reduced to study asymptotic properties of solutions to a Poisson equation. The details are as follows. We define $\phi(x) = \P[R > x]$ and $D(x) = e^{\a x} \phi(e^x)$. Our aim is to prove
$$
\lim_{x \to \infty} D(x) = C.
$$

\begin{lem}
The function $D$ satisfies the following Poisson equation

\begin{equation}\label{eq:pois}
\E[D(x + Y)] = D(x) + G(x),
\end{equation}
where
$$
G(x) = e^{\a x} \E\left[\sum_{i=1}^N \phi \left(\frac{e^x}{A_i}\right) -1 + \mathbbm{1}(B\leq e^x)\prod_{i=1}^N\left(1-\phi \left(\frac{e^x}{A_i}\right) \right) \right]
$$
and $Y$ is the random variable defined in \eqref{eq:3}.
\end{lem}
\begin{proof} We start with finding a recursive formula for $\phi$. For this purpose we
denote by $\mu$ the distribution of $(N,B,A_1,A_2,...)$ and write
\begin{equation*}
\begin{split}
\phi(x) & = \P[R > x] = \P\left[\bigvee_{i=1}^N A_i R_i \vee B > x\right]  = 1-\P\left[\bigvee_{i=1}^N A_i R_i \vee B \leq x\right]\\
&= 1-\P\left[A_1 R_1 \leq x,..., A_N R_N \leq x,B\leq x \right] \\
& = 1- \int \P\left[R_1 \leq \frac{x}{a_1},..., R_n \leq \frac{x}{a_n},b\leq x \right] d\mu(n,b, a_1, a_2,...)\\
&= 1- \int \mathbbm{1}(b\leq x) \prod_{i=1}^n \left(1-\phi\left( \frac{x}{a_i} \right)\right) d\mu(n,b, a_1, a_2,...) \\
& = 1- \E\left[\mathbbm{1}(B\leq x) \prod_{i=1}^N \left(1-\phi\left( \frac{x}{A_i} \right)\right) \right].
\end{split}
\end{equation*}
By the definition of $Y$ and the many-to-one formula \eqref{eq:5} we obtain
\begin{equation*}
\begin{split}
\E[D(x + Y)]  & =  \E \left[e^{\a (x+Y)} \phi(e^{x+Y}) \right] \\ & = e^{\a x}\E \left[\sum_{i=1}^N e^{-\a \log{A_i}} \phi(e^{x-\log{A_i}})  A_i^{\a}\right] \\ &= e^{\a x}\E \left[\sum_{i=1}^N \phi\left(\frac{e^x}{A_i} \right) \right].
\end{split}
\end{equation*}
Therefore
\begin{equation*}
\begin{split}
\E[D(x + Y)] - D(x) &= e^{\a x}\E \left[\sum_{i=1}^N \phi\left(\frac{e^x}{A_i} \right) \right] - e^{\a x} \left(1- \mathbbm{1}(B\leq e^{x})\E\left[\prod_{i=1}^N \left(1-\phi\left( \frac{e^x}{A_i} \right)\right) \right] \right)\\ &= G(x).
\end{split}
\end{equation*}
\end{proof}

We now show some properties of function $G$.

\begin{lem}\label{lem:4.2}
Assume \ref{st:5}. Then
$$
\lim_{x \to \infty} G(x) = 0.
$$
\end{lem}

\begin{proof}
We decompose $G$ as a sum of two functions
\begin{equation} \label{eq:31}
\begin{split}
G(x) & = e^{\a x} \E\left[\sum_{i=1}^N \phi \left(\frac{e^x}{A_i}\right) -1 + \prod_{i=1}^N\left(1-\phi \left(\frac{e^x}{A_i}\right) \right) \right] - e^{\a x} \E\left[  \mathbbm{1}(B> e^{x}) \prod_{i=1}^N\left(1-\phi \left(\frac{e^x}{A_i}\right) \right) \right]\\
&=f_1(x) - f_2(x).
\end{split}
\end{equation}
Notice that $f_1$ is positive. Indeed, it is sufficient to apply
the following inequality, valid for $0 \leq u_i \leq v_i \leq 1$ (see \cite{DL}, p. 283):
$$
\prod_{i=1}^n u_i -1 + \sum_{i=1}^{n} (1-u_i) \geq \prod_{i=1}^n v_i -1 + \sum_{i=1}^{n} (1-v_i)
$$
with $u_i = 1-\phi \left(\frac{e^x}{A_i}\right)$ and $v_i =1$.

We first show that $f_1(x)$ tends to $0$. For this purpose recall an easy inequality
$$
u \leq e^{-(1-u)},
$$
valid for any real $u$ and write
\begin{equation*}
\begin{split}
 f_1(x) & \leq e^{\a x} \E\left[\sum_{i=1}^N \phi \left(\frac{e^x}{A_i}\right) -1 + \prod_{i=1}^N\left(e^{-\phi \left(\frac{e^x}{A_i}\right)} \right) \right] = e^{\a x} \E\left[\sum_{i=1}^N \phi \left(\frac{e^x}{A_i}\right) -1 + e^{-\sum_{i=1}^N \phi \left(\frac{e^x}{A_i}\right)} \right] \\
& = e^{\a x} \E\left[F \left( \sum_{i=1}^N \phi \left(\frac{e^x}{A_i}\right) \right)  \right],
\end{split}
\end{equation*}
where $F(u) = e^{-u} -1 + u$. Observe that the function $F$ is increasing on $[0,\infty)$, therefore by Lemma \ref{prop1}.
$$
e^{\a x} \E\left[F \left( \sum_{i=1}^N \phi \left(\frac{e^x}{A_i}\right) \right)  \right]  \leq e^{\a x} \E\left[F \left( e^{-\a x} \sum_{i=1}^N A_i^{\a} \right)  \right].
$$
Note that $H(u) = \frac{F(u)}{u}$ is bounded and tends to $0$ as $u \to 0$. These observations and the dominated convergence theorem give us

\begin{equation*}
\begin{split}
\limsup_{x \to \infty} f_1(x) & \leq \limsup_{x \to \infty} e^{\a x} \E\left[  F \left( e^{-\a x} \sum_{i=1}^N A_i^{\a} \right)  \right]\\
&= \limsup_{x \to \infty} e^{\a x} \E\left[\frac{F \left( e^{-\a x} \sum_{i=1}^N A_i^{\a} \right) }{e^{-\a x} \sum_{i=1}^N A_i^{\a} } e^{-\a x} \sum_{i=1}^N A_i^{\a}  \right]\\
& = \limsup_{x \to \infty} \E\left[H\left( e^{-\a x} \sum_{i=1}^N A_i^{\a} \right)  \sum_{i=1}^N A_i^{\a} \right] \\ &= \limsup_{t \to 0} \E\left[H\left( t \sum_{i=1}^N A_i^{\a} \right)  \sum_{i=1}^N A_i^{\a} \right]=0.
\end{split}
\end{equation*}
To bound $f_2$ we use Chebyshev's inequality with $\a < \beta < \a + \delta $
$$
f_2(x)=e^{\a x} \E\left[  \mathbbm{1}(B> e^x) \prod_{i=1}^N\left(1-\phi \left(\frac{e^x}{A_i}\right) \right) \right] \leq e^{\a x}  \P(B> e^x) \leq \E\big[B^{\beta}\big] e^{x(\a - \beta)} \to 0,
$$
as $x \to \infty$.
\end{proof}

\begin{lem}\label{lem:5.6}
Assume \ref{st:5}. There is $\epsilon > 0$ such that $e^{\epsilon |x|}G(x) \in L^1(\mathbb{R})$.
\end{lem}
\begin{proof}
Once more we use decomposition \eqref{eq:31}.
Take any $0 < \epsilon < \min(\a/2, \delta)$. Let us first consider function $e^{\epsilon |x|}f_1(x)$. To show integrability on $(\infty, 0]$, recall that $f_1$ is positive and  use Chebyshev's inequality with $\epsilon$
\begin{equation*}
\begin{split}
e^{-\epsilon x} f_1(x) & \leq e^{(\a-\epsilon) x} \E\left[\sum_{i=1}^N \phi \left(\frac{e^x}{A_i}\right) \right] +e^{(\a-\epsilon) x} \\ &\leq e^{(\a -\epsilon)x} \E\bigg[\sum_{i=1}^N e^{-\epsilon x} A_i^{\epsilon} \bigg] +  e^{(\a -\epsilon)x} \\&\le Ce^{(\a-2\epsilon)x},
\end{split}
\end{equation*}
hence the integral $\int_{-\infty}^{0} e^{-\epsilon x} f_1(x) dx$ is finite. To deal with the right tail we use the fact that  $F$ is an increasing function on $[0,\infty)$ and $F(u) \leq u$. Choose  $\beta$ such that $\frac{3}{4} \a < \beta < \a$. Again using Chebyshev's inequality we write
\begin{equation*}
\begin{split}
\int_0^{\infty}e^{\epsilon x} f_1(x) dx & \leq \int_0^{\infty} e^{(\a + \epsilon) x} \E\left[F\left(\sum_{i=1}^N \phi \left(\frac{e^x}{A_i}\right) \right) \right] dx\\
&\leq \int_0^{\infty} e^{(\a+\epsilon) x} \E\left[F\left(\sum_{i=1}^N \E[R^{\beta}] e^{-\beta x} A_i^{\beta} \right) \right] dx\\
& =\E\left[ \int_0^{\infty} e^{(\a+\epsilon) x} F\left(\E[R^{\beta}] e^{-\beta x} \sum_{i=1}^N  A_i^{\beta} \right) dx \right],
\end{split}
\end{equation*}
where the last equality holds by Fubini's theorem.

We now use a substitution $u = \E[R^{\beta}] e^{-\beta x} \sum_{i=1}^N  A_i^{\beta}$ and again by Fubini's theorem we obtain
\begin{equation*}
\begin{split}
\E\left[ \int_0^{\infty} e^{(\a+\epsilon) x} F\left(\E[R^{\beta}] e^{-\beta x} \sum_{i=1}^N  A_i^{\beta} \right)dx  \right]  & \leq \E \left[\int_{0}^{\infty} \frac{1}{\beta} \left(C \sum_{i=1}^{N} A_{i}^{\beta} \right)^{\frac{\a + \epsilon}{\beta}} \frac{F(u)}{u^{1+\frac{\a + \epsilon}{\beta}}}du \right]\\
& = C \E \left[ \left( \sum_{i=1}^{N} A_{i}^{\beta} \right)^{\frac{\a + \epsilon}{\beta}} \right] \int_{0}^{\infty} \frac{F(u)}{u^{1+\frac{\a + \epsilon}{\beta}}}du.
\end{split}
\end{equation*}
To show that the above is finite, we write
$$
\int_{0}^{\infty} \frac{F(u)}{u^{1+\frac{\a + \epsilon}{\beta}}}du = \int_{0}^{1} \frac{F(u)}{u^{1+\frac{\a + \epsilon}{\beta}}}du + \int_{1}^{\infty} \frac{F(u)}{u^{1+\frac{\a + \epsilon}{\beta}}}du.
$$
To estimate the first integral we only need to bound integrand near zero. To obtains this, it is sufficient to observe that $\lim_{u \to \infty} \frac{F(u)}{u^2} = \frac{1}{2}$ and our assumptions on $\beta$ and $\epsilon$ imply $\frac{\a + \epsilon}{\beta} < 2$. For the second integral notice that $F(u) \leq u$ for any $u \geq 0$, therefore
$$
\int_{1}^{\infty} \frac{F(u)}{u^{1+\frac{\a + \epsilon}{\beta}}} du < \infty.
$$
For the expectation factor we use the inequality
$$
\E \left[\left(\sum_{i=1}^{N}X_{i}^{1/r}  \right)^{p} \right] \leq C_{r,p}\E\left[\sum_{i=1}^{N} X_{i} \right],
$$
valid, under assumption $\E [N^{1+\delta}]<\8$ for   any sequence of positive random variables  $\{X_i\}$,  $r > 1$ and  $p \in (1,\frac{r(1+\delta)}{r+\delta})$ (see \cite{BK}, Lemma 3.4). Plugging $r = \frac{\a + \epsilon}{\a}$ and $X_{i} = A_{i}^{r \beta}$ we obtain
$$\E \left[ \left( \sum_{i=1}^{N} A_{i}^{\beta} \right)^{\frac{\a + \epsilon}{\beta}} \right] < \8.$$

Integrability of $e^{\epsilon |x|} f_2(x)$ comes easily from Chebyshev's inequality. Indeed, once more take $\a + \epsilon < \beta < \a + \delta$
$$
 e^{\epsilon |x|} f_2(x) \leq e^{\epsilon |x|} e^{\a} \P\left[B > e^x \right] \leq C \min(e^{x(\a - \epsilon)},e^{x(\a + \epsilon - \beta)}).
$$
\end{proof}

Our aim is to deduce some asymptotic properties of the function $D$, knowing that it is a solution to the Poisson equation \eqref{eq:pois} for some well behaved function $G$. A typical argument reduces the problem to the key renewal theorem, which in turn requires $G$ to be directly Riemann integrable (see \cite{F} for the precise definition). For this purpose we need to prove some local properties of $G$ and this cannot be done directly. To avoid this problem we proceed as in Goldie's paper \cite{G}
and for an integrable function $f$ we define the smoothing operator
$$
\breve{f}(x) = \int_{-\infty}^x e^{-(x-u)} f(u) du.
$$
Note that $f(x) \lessgtr M$ implies $\breve{f}(x) \lessgtr M$, $\lim_{x \to \pm \infty} f(x) = 0$ implies $\lim_{x \to \pm \infty} \breve{f}(x) = 0$ and $\int_{\mathbbm{R}} \breve{f}(x) dx = \int_{\mathbbm{R}} f(x) dx$. Moreover, $\breve{f}$ is always continuous function and if $f$ is integrable, then $\breve{f}$ is directly Riemann integrable (dRi) (see Goldie \cite{G}, Lemma 9.2).

Smoothing both sides of equation \eqref{eq:pois} we obtain
\begin{equation}\label{eq:pois2}
\E[\breve D(x + Y)] = \breve D(x) + \breve G(x).
\end{equation}
Notice that $\breve G$ has now better properties than $G$. Below we will describe asymptotic behavior of $\breve D$ and finally deduce the main result.

Define
$$
\overline{f}(x) = \int_{-\infty}^{x} \breve{f}(s)ds.
$$
The following lemma holds.

\begin{lem}
For a given function $f$ suppose that there is a positive $\epsilon$ such that $e^{\epsilon |x|} f(x) \in L^1(\mathbb{R})$. If $\int_{\mathbb{R}} {f}(s)ds = 0$ then $\overline{f}$ is dRi and $$\int_{\mathbb{R}} \overline{f}(s) ds = -\int_{\mathbb{R}} s\breve{f}(s) ds.$$
\end{lem}
\begin{proof}
We only show the first property, since the second one is a consequence of the integration by parts. One has
\begin{equation*}
\begin{split}
\overline{f}(x) & = \int_{-\infty}^{x} \int_{-\infty}^{u} e^{-(u-s)} f(s)ds\: du = \int_{-\infty}^{x}e^s f(s) \int_{s}^{x} e^{-u} du\: ds\\
&= \int_{-\infty}^{x}e^s f(s)(e^{-s} - e^{-x}) ds   = \int_{-\infty}^{x}f(s) ds - \breve{f}(x).
\end{split}
\end{equation*}
Thus
$$
|\overline{f}(x)| \leq \left| \int_{-\infty}^{x}f(s) ds \right| + | \breve{f}(x) |.
$$
For $x \leq 0$ we have
$$
|\overline{f}(x)| \leq \int_{-\infty}^{x} \left| f(s)  \right|ds + | \breve{f}(x) | \leq \int_{-\infty}^{x} e^{\epsilon x}e^{-\epsilon s} \left| f(s)  \right|ds
 + | \breve{f}(x)| \leq | \breve{f}(x)|
+ C e^{\epsilon x}.
$$
Recall that $\int_{\mathbb{R}} \breve{f}(s) ds = \int_{\mathbb{R}} f(s) ds$. Similarly to the above, for $x \geq 0$ one has
$$
|\overline{f}(x)| =\left|\int_{-\infty}^{x}f(s) ds - \breve{f}(x) \right| = \left|\breve{f}(x)+ \int_{x}^{\infty}f(s) ds \right| \leq |\breve{f}(x)| + \left| \int_{x}^{\infty}f(s) ds \right| \leq  |\breve{f}(x)| + Ce^{-\epsilon x},
$$
hence $\overline{f}$ is dRi as a function bounded by dRi functions.
\end{proof}

\begin{prop}\label{prop:4}
For any $y \in \mathbb{R}$ we have
\begin{equation}\label{eq:16}
\lim_{x \to \infty}\frac{\breve{D}(x+y)}{\breve{D}(x)} = 1.
\end{equation}
\end{prop}
\begin{proof}
Take $K$ big enough to ensure that for any $x \geq K$ we have  $D(x) > 0$. We define a family $\{h_x\}_{x \geq K}$ of continuous functions by
$$
h_x(y) = \frac{\breve{D}(x+y)}{\breve{D}(x)}.
$$
Such a family and its properties was already considered e.g. in \cite{BK,DL,L}, with a slightly different definition of function $D$, though.
Notice that the family $\{h_x\}_{x \geq K}$ is uniformly bounded and equicontinuous. Indeed, boundedness  is straightforward since by \eqref{eq:7} and Lemma \ref{prop1} one has $D(x) \leq C$ and $D(x) > \frac{1}{C} > 0$ for $x$ sufficiently large and hence the same holds for $\breve{D}$. To obtain equicontinuity, for $h >0$, we write
\begin{equation*}
\begin{split}
|h_x(y+h) - h_x(y)| & \leq \frac{1}{C} \left| \int_{-\infty}^{x+y+h} e^{-(x+y+h-u)} D(u) du - \int_{-\infty}^{x+y} e^{-(x+y-u)} D(u) du \right| \\
& = \frac{1}{C} \left| \int_{x+y}^{x+y+h} e^{-(x+y+h-u)} D(u) du - \int_{-\infty}^{x+y} e^{-(x+y-u)}D(u)(1-e^{-h})  du \right| \\
& \leq \frac{1}{C} \left( h + 1-e^{-h} \right) \to 0 \quad \text{as } h \to 0
\end{split}
\end{equation*}
and the very last expression does not depend on $x$ (it does not depend on $y$ as well, so we obtain even uniform equicontinuity).

In view of the  Arzel\`{a}-Ascoli theorem,  the family $\{h_x\}_{x \geq K}$ is relatively compact in the topology induced by the uniform norm on compact sets. We conclude that there is a sequence ${x_n} \to \infty$ and $h$ such that $h_{x_n} \to h$ uniformly.

Using \eqref{eq:pois2} we write
$$
\breve{D}(x+y) = \E[\breve{D}(x +y + Y)] - \breve{G}(x+y),
$$
which we divide by $\breve{D}(x)$, to obtain
\begin{equation}\label{eq:22}
h_x(y) = \E\left[h_x(y+Y)\right] - \frac{ \breve{G}(x+y)}{\breve{D}(x+y)} h_x(y).
\end{equation}

By Lemma \ref{prop1} we have $D(x) > C$ and also $\breve D(x)>0$ for any sufficiently large $x$. Thus, by Lemma \ref{lem:4.2} we see that
\begin{equation}\label{eq:12}
\lim_{x \to \infty} \frac{\breve{G}(x)}{\breve{D}(x)} = 0.
\end{equation}

Passing with $x_n \to \infty$ in \eqref{eq:22}, using \eqref{eq:12} and dominated convergence theorem yield
$$
h(y) = \E\left[h(y+Y)\right].
$$
As a consequence of Choquet-Deny theorem (see e.g. \cite{R}, Theorem 1.3 in Chapter 5), any positive $Y$-harmonic function is constant, thus we see that $h(y) = h(0) = 1$. It implies that $h$ is the unique accumulation point and \eqref{eq:16} holds.
\end{proof}

Now we are ready to prove main results.

\begin{proof}[Proof of Theorem \ref{th1}]
Let ${Y_i}$ be an i.i.d sequence with distribution defined by \eqref{eq:3}. Denote $S_n = \sum_{i=1}^n Y_i$ (we assume $S_0 = 0$). Define also stopping times $L = \inf \{n \geq 0 : S_n < 0\}$ and $T_k = \inf \{n > T_{k-1} : S_n \geq S_{T_{k-1}} \}$, $T_0 = 0$. Equation \eqref{eq:pois2} implies that for any $x$ the process
$$
M_n(x) = \breve{D}(x + S_n) - \sum_{i=0}^{n-1}\breve{G}(x + S_i)
$$
forms a martingale with respect to the natural filtration generated by ${Y_i}$. By the optional stopping theorem
$$
\E[M_{n \wedge L }] = \E[M_0(x)] = \breve{D}(x),
$$
or in another words
\begin{equation}\label{eq:15}
\E[\breve{D}(x + S_{n \wedge L})] -\breve{D}(x) = \E\left[\sum_{i=0}^{n \wedge L-1} \breve{G}(x+S_i)\right].
\end{equation}
By the duality principle (see \cite{F}, Chap. XII)
$$
\E\left[\sum_{i=0}^{ L-1}| \breve{G}(x+S_i)| \right] = \E\left[\sum_{i=0}^{\infty} | \breve{G}(x+S_{T_i}) |\right]
$$
and the right hand side series is finite since $\breve{G}$ is dRi. From Proposition \ref{prop:4} we conclude
$\breve{D}(x) \leq C e^{\epsilon |x|}$ and since $\E\left[e^{\epsilon S_L} \right] < \infty$ (see \cite{F}, Chap. XII) we can pass with $n$ to infinity in \eqref{eq:15} to get
\begin{equation}\label{eq:13}
\E[\breve{D}(x + S_{L})] -\breve{D}(x) = \E\left[\sum_{i=0}^{L-1} \breve{G}(x+S_i)\right] =: R(x).
\end{equation}
Again by the duality principle we have
$$
R(x) = \E\left[\sum_{i=0}^{\infty} \breve{G}(x+S_{T_i})\right].
$$
The key renewal theorem now yields
$$
\lim_{x \to \infty} R(x) = -\frac{\int_{\mathbb{R}} \breve{G}(x) dx}{\E[S_{T_1}]}.
$$
Integrating \eqref{eq:13} one has
$$
\int_0^x\left(  \E[\breve{D}(y + S_{L})] -\breve{D}(y) \right) dy  = \int_0^x R(y) dy,
$$
which we can rewrite as
\begin{equation}\label{eq:14}
\breve{D}(x) \cdot \E\left[ \int_0^{S_L} \frac{\breve{D}(x+y)}{\breve{D}(x)} dy \right] - \E \left[\int_0^{S_L} \breve{D}(y) dy \right]  = \int_0^x R(y) dy.
\end{equation}
The same argument as before allows us to pass with $x$ to the infinity under the integral sign, hence we have
$$
\lim_{x \to \infty} \E\left[ \int_0^{S_L} \frac{\breve{D}(x+y)}{\breve{D}(x)} dy \right] = \E[S_L] .
$$
Now we divide \eqref{eq:14} by $x$ and pass to the limit, which implies
$$
\lim_{x \to \infty} \frac{\breve{D}(x)}{x} = \lim_{x \to \infty} \frac{R(x)}{\E[S_L]} = -\frac{\int_{\mathbb{R}} \breve{G}(x) dx}{\E[S_{T_1} ]\E[S_L]}.
$$
The upper bound on $\breve{D}(x)$ show that above limit equals $0$, hence
$$
\int_{\mathbb{R}} \breve{G}(x) dx = 0.
$$
This procedure may be repeated using integrating \eqref{eq:13} with different limits, i.e. we write
$$
\int_{-\infty}^x\left(  \E[\breve{D}(y + S_{L})] -\breve{D}(y) \right) dy  = \int_{-\infty}^x \E\left[\sum_{i=0}^{\infty} \breve{G}(y+S_{T_i})\right] dy,
$$
or equivalently
$$
\breve{D}(x)\cdot\E\left[ \int_0^{S_L} \frac{\breve{D}(x+y)}{\breve{D}(x)}  dy \right] =\E\left[  \sum_{i=0}^{\infty} \int_{-\infty}^x \breve{G}(y+S_{T_i}) dy\right] = \E\left[  \sum_{i=0}^{\infty}  \overline{G}(x+S_{T_i}) \right].
$$
We pass with $x$ to the infinity and again by renewal theorem obtain
$$
\lim_{x \to \infty} \breve{D}(x) = -\frac{\int_{\mathbb{R}} \overline{G}(x) dx}{\E[S_{T_1} ]\E[S_L]} = \frac{\int_{\mathbb{R}} x\breve{G}(x) dx}{\E[S_{T_1} ]\E[S_L]}.
$$
Finally, observe that
\begin{equation*}
\begin{split}
\lim_{x \to \infty} \breve{D}(x) &= \lim_{x \to \infty} \int_{-\infty}^{x} e^{-(x-u)} D(u) du = \lim_{x \to \infty} e^{-x}\int_{0}^{e^x}  D(\log{t}) dt\\
&= \lim_{x \to \infty} \frac{1}{x}\int_{0}^{x}  D(\log{t}) dt = \lim_{t \to \infty} t^{\a} \P[R > t],
\end{split}
\end{equation*}
where the last equality is a consequence of Lemma 9.3 in \cite{G}.
\end{proof}

\appendix
\section{Lower estimates}
\label{appendix}

\begin{lem}
Assume hypotheses of Theorem \ref{th1} and suppose additionally that $N$ is constant and $B=1$ a.s. Then there exists a constant $C>0$ such that
$$
\P[R>t] \ge C t^{-\a}
$$ for sufficiently large $t$.
\end{lem}
\begin{proof}
The arguments presented below base on \cite{A}.  
For $v \in \mathcal{T}$ we define $\tau_t(v) = \inf\{1 \leq k \leq |v| : L(v_k) > t\}$ and for $S_n$ as in the proof of Lemma \ref{prop1} denote $\tau_t = \inf \{n : S_n < -\log t \}$. Moreover denote $Z_t = \{v \in \mathcal{T} : \tau_t(v) = |v| \}$ and let $N_t = \sum_{v \in \mathcal{T}} \mathbbm{1}\left(v \in Z_t \right)=\sum_{n} \sum_{|v|=n} \mathbbm{1}\left(v \in Z_t \right)$ be the number of elements of $Z_t$. Note that an event $\{N_t > 0\}$ is contained in  $\{ \bigvee_{v \in \mathcal{T}} L(v) > t \}$ hence it is sufficient to prove the lower bound for $\P[N_t > 0]$. For this purpose we use the so-called second moment method, i.e. by the Cauchy-Schwarz inequality one has
$$
(\E[N_t])^2 = (\E[N_t \mathbbm{1}(N_t > 0)])^2 \leq \E[N_t^2] \P[N_t > 0],
$$
which implies
$$
\P[N_t > 0] \geq \frac{(\E[N_t])^2 }{\E[N_t^2]}.
$$
For $x \geq 1$ we call $\P^x$ the distribution such that
$\P^x \left[L(\varnothing) = x\right] = 1$ (hence $\P = \P^1$) and $\E^x$ the corresponding expectation. First, we will show that there is a positive constant $C$ such that for sufficiently large $t$ the following holds
\begin{equation}\label{eq:4.3}
C \leq t^\a x^{-\a} \E^x[N_t] \leq 1.
\end{equation}
We define for any $s \geq 1$ the undershoot $L_s = -\log s - S_{\tau_{s}}$. Using the many-to-one formula \eqref{eq:5} we obtain
\begin{equation*}
\begin{split}
\E^x[N_t] &= \E^x \left[\sum_{v \in \mathcal{T}}\mathbbm{1}\left(\tau_t(v) = |v| \right) \right]  = \sum_n \E^x \left[\sum_{|v|=n} \mathbbm{1}\left(\tau_t(v) = |v| \right) \right] = \sum_n \E^x \left[e^{\a S_n} \mathbbm{1}\left(\tau_t =n \right) \right]\\
&= (t/x)^{-\a} \E \left[e^{\a (S_{\tau_{t/x}} + \log (t/x))} \right] = (t/x)^{-\a} \E \left[e^{-\a L_{t/x}} \right].
\end{split}
\end{equation*}
 Obviously, for any $s$ one has $\E \left[e^{-\a L_s}\right] \leq 1$, since $L_s > 0$. For the lower bound note that for any positive $M$
$$
 \E\left[e^{-\a L_s} \right]\geq  e^{-\a M}\P\left[L_s \leq M \right].
$$
Choosing appropriately large $M$, there is a constant $C$ such that  $\P\left[L_s \le  M \right] >C> 0$, uniformly with respect to $s$ (see \cite{CH}, Proposition 4.2) which shows \eqref{eq:4.3}.

Let us define $N_t(n) = \sum_{|v|=n} \mathbbm{1}(v \in Z_t)$ and $N_t(\leq n) = \sum_{k\leq n}\sum_{|v|=k} \mathbbm{1}(v \in Z_t)$. We have

$$
\E[N_t^2] \leq 2 \sum_n \E \left[N_t(n) \sum_{k\leq n} N_t(k) \right]  = 2 \sum_n \E \left[\sum_{|v|=n} \mathbbm{1}(v \in Z_t) N_t(\leq n) \right].
$$
To estimate $\E \left[\mathbbm{1}(v \in Z_t) N_t(\leq n) \right]$ we decompose $N_t(\leq n)$ along the vertex $v$, i.e. we write
$$
N_t(\leq n) = 1 + \sum_{k=0}^{n-1} N_t^{v_k},
$$
where $N_t^{v_k}$ is the number of descendants $u$ of $v_k$ at the level at most $n$ which are not descendants of $v_{k+1}$ and such that $u$ is an element of $Z_t$. Denote by $\mathcal{F}_n$ the $\sigma$-algebra generated by the tree up to the level $n$. The above implies that
\begin{equation*}
\begin{split}
\E \left[N_t^{v_k}| \mathcal{F}_k \right]& \leq (N-1)\left(\E^{L(v_k)} \left[\mathbbm{1}(A(u_1) \leq t) N_t \right] + \P^{L(v_k)} \left[ A(u_1) > t \right] \right) \\
& \leq (N-1) \left(C_1 \E^{L(v_k)} \left[N_t \right] + C_2 L^{\a}(v_k) t^{-\a} \right) \\
& \leq C t^{-\a} L^{\a}(v_k),
\end{split}
\end{equation*}
 hence using \eqref{eq:5}
\begin{equation*}
\begin{split}
\E[N_t^2] & \leq C\left(\E[N_t] + t^{-\a} \sum_n \E \left[\sum_{|v|=n} \sum_{k=0}^{n-1} \mathbbm{1}(v \in Z_t) L^{\a}(v_k) \right]  \right)\\
& \leq C_1 t^{-\a} + C_2 t^{-\a} \E \left[ \sum_n e^{\a S_n} \sum_{k=0}^{n-1} e^{-\a S_k} \mathbbm{1}(\tau_t = n) \right]\\
& \leq C_1 t^{-\a} + C_2 t^{-\a} \E \left[ \sum_{k=0}^{\tau_t-1} e^{-\a (S_k+ \log t)} \right]\\
&  = C_1 t^{-\a} + C_2 t^{-\a} W(\log t)\\
& \leq C t^{-\a},
\end{split}
\end{equation*}
 for the bounded function $W$ defined by \eqref{eq:W}.
Now we see that $\P[N_t > 0] \geq \frac{(\E[N_t])^2 }{\E[N_t^2]} \geq C t^{-\a}$.
\end{proof}

\end{document}